\documentclass[12pt]{article}
\usepackage{hyperref}
\newcommand{\footremember}[2]{%
    \footnote{#2}
    \newcounter{#1}
    \setcounter{#1}{\value{footnote}}%
}

\author{%
  Fabián Arias\footremember{alley}{Facultad de Ciencias Básicas, Universidad Tecnológica de Bolívar, Colombia; email address: \tt{farias@utb.edu.co}, ORCID 0000-0002-2944-0427}%
  \and Jerson Borja\footremember{trailer}{Departamento de Matemáticas, Universidad de Córdoba, Colombia; email address: \tt{jersonborjas@correo.unicordoba.edu.co}, ORCID 0000-0003-4366-4328}%
  \and Calixto Rhenals \footnote{Departamento de Matemáticas, Universidad de Córdoba, Colombia; email address: \tt{calixtorhenalsj@correo.unicordoba.edu.co}}%
  }
\date{}
\usepackage[utf8]{inputenc}
\usepackage[english]{babel}
\usepackage{amssymb}
\usepackage{amsthm}
\usepackage{amsmath}
\allowdisplaybreaks
\usepackage{color}

\newtheorem{proposition}{Proposition}[section]
\newtheorem{lemma}[proposition]{Lemma}
\newtheorem{theorem}[proposition]{Theorem}
\newtheorem*{theorem*}{Theorem}
\newtheorem{corollary}[proposition]{Corollary}
\theoremstyle{definition}
\newtheorem{remark}[proposition]{Remark}

\newtheorem{example}[proposition]{Example}

\providecommand{\keywords}[1]
{
  \small	
  \textbf{\textit{Keywords---}} #1
}

\title{The Frobenius problem for numerical semigroups generated by tails of sequences of the form $ca^n-d$}

\date{}

\begin{document}
\maketitle

\begin{abstract}
For a positive integer $n$, we consider the submonoid $S_n$ of $\mathbb N$ generated by the sequence of positive integers $\mathbf s_j=ca^{n+j}-d$, $j\in \mathbb N$, where $a, c$ and $d$ are integer, $a\geq 2$ and $c$ is positive.

In this paper, we solve the Frobenius problem for the numerical semigroup $(1/\lambda)S_n$, where $\lambda=\gcd(S_n)$, under fairly general conditions on $a, c$ and $d$. We give conditions to determine the embedding dimension and the minimal generating set of $S_n$, and we also characterize the maximum element of the Apéry set ${\rm Ap}(S_n, \mathbf s_0)$.   
\end{abstract}


\section{Introduction}

\paragraph{Numerical semigroups and submonoids of $\mathbb N$.}
Let $\mathbb N=\{0,1,\ldots\}$ be the set of natural numbers and $\mathbb Z$ be the set of integer numbers. A numerical semigroup is a subset $S$ of $\mathbb N$ that is closed under addition, $0\in S$, and $\mathbb N\setminus S$ is a finite set. For a subset $A$ of $\mathbb N$, we denote by $\langle A\rangle$ the submonoid of $\mathbb N$ generated by $A$: 
\begin{equation*}
    \langle A\rangle=\left\{\sum_{j=1}^{n}c_j\mathbf a_j:n\in\mathbb N\setminus\{0\}, c_j\in\mathbb N, \mathbf a_j\in A\text{ for }j=1,\ldots, n\right\}.
\end{equation*}
It is known that $\langle A\rangle$ is a numerical semigroup if and only if ${\gcd}(A)=1$ (see \cite{rosales}). 

Let $M$ be a nontrivial submonoid of $\mathbb N$. Then, $M$ is finitely generated and, moreover, there exists a subset of $M$ that generates $M$ such that any other generating set of $M$ contains it; such generating set is called the \textit{minimal generating set} of $M$ and its elements, which are called the \textit{minimal generators} of $M$, are precisely the nonzero elements belonging to $M$ that cannot be written as the sum of two nonzero elements in $M$. We denote by $e(M)$ the number of minimal generators of $M$.  

If $S$ is a numerical semigroup, we call $e(S)$ the \textit{embedding dimension} of $S$. The maximum element in $\mathbb N\setminus S$ is called the \textit{Frobenius number of $S$}, which is denoted by ${\rm F}(S)$ and the cardinality of $\mathbb N\setminus S$ is the \textit{genus of $S$}, which is denoted by ${\rm g}(S)$. 

Assume that $M$ is a nontrivial submonoid of $\mathbb N$. If $x\in M\setminus\{0\}$, the \textit{Apéry set of $x$ in $M$}, denoted by ${\rm Ap}(M,x)$, is defined by ${\rm Ap}(M,x)=\{m\in M:m-x\notin M\}$. If $d=\gcd(M)$ and $S$ is the numerical semigroup obtained by dividing every element in $M$ by $d$, then we have the relation ${\rm Ap}(M,x)=\{ds:s\in {\rm Ap}(S,x/d)\}$. If $S$ is a numerical semigroup and $x\in S\setminus\{0\}$, then ${\rm Ap}(S,x)=\{w(0), \ldots, w(x-1)\}$, where $w(i)$ is the least element in $S$ congruent to $i$ modulo $x$; we also have the following two properties:
\begin{enumerate}
    \item ${\rm F}(S)=\max {\rm Ap}(S, x)-x$, 
    \item $\displaystyle{\rm g}(S)=\frac{1}{x}\left(\sum_{s\in {\rm Ap}(S, x)}s\right)-\frac{x-1}{2}$.
\end{enumerate}

An integer number $x$ is a \textit{pseudo-Frobenius number} of $S$ if $x\notin S$ and $x+s\in S$ for all $s\in S\setminus\{0\}$. The set of pseudo-Frobenius numbers of $S$ is denoted by ${\rm PF}(S)$. We have ${\rm PF}(S)$ is nonempty, as ${\rm F}(S)$ is a pseudo-Frobenius number of $S$. A partial order $\leq_S$ is defined in $S$ as follows: for $a, b\in S$, $a\leq_Sb$ if and only if $b-a\in S$. If $x\in S\setminus\{0\}$, then Pseudo-Frobenius numbers of $S$ are precisely those integers of the form $w-x$, where $w$ is a maximal element in ${\rm Ap}(S, x)$ with respect to $\leq_S$, that is 
\begin{equation*}
    {\rm PF}(S)=\left\{w-x: w\in {\rm Maximals}_{\leq_S}{\rm Ap}(S,x)\right\}.
\end{equation*}
The set ${\rm PF}(S)$ is finite, and its cardinality is called \textit{the type} of $S$, which is denoted by ${\rm t}(S)$. All these concepts and results can be found in \cite{rosales}.

\paragraph{Submonoids and numerical semigroups generated by tails of sequences.}
Let us consider a sequence $(x_n)_{n\geq 1}$ of positive integers, which we will call the \textit{generating sequence}. For each $n\geq 1$, the $n-$tail of $(x_n)_{n\geq 1}$ is the subsequence $(x_n, x_{n+1}, \ldots)$, and we define a submonoid $S_n$ of $\mathbb N$ by $S_n=\langle\{x_{n+j}:j\in\mathbb N\}\rangle$. As we know, $S_n$ is a numerical semigroup if and only if $\gcd(x_n, x_{n+1},\ldots)=1$. In this way, we associate a family of submonoids of $\mathbb N$ with a given sequence of positive integers. In recent works on numerical semigroups \cite{gu*, gu, mersenne, repunit, thabit, songthabit, song}, a particular generating sequence $(x_n)_{n\geq 1}$ is given, and for the numerical semigroups $S_n$ associated with this sequence, they find
\begin{enumerate}
    \item the minimal generating set and the embedding dimension of $S_n$;
    \item the Apéry set ${\rm Ap}(S_n, x_n)$, the Frobenius number and genus of $S_n$;
    \item in most cases, the pseudo-Frobenius numbers and the type of $S_n$ are computed. 
\end{enumerate}
The generating sequences considered in those works have the form $x_n=ca^n-d$, for particular values of the integers $a, c$ and $d$, where $a\geq 2$ and $c$ is positive.

\paragraph{Presentation of the results.} We give arbitrary integers $a, c$ and $d$, with $a\geq 2$ and $c>0$. Then, we consider the generating sequence $(x_n)_{n\geq 1}$ given by
\begin{equation*}\label{eq-sequence_x_n}
x_n=ca^n-d. 
\end{equation*}
One first condition on $a, c$ and $d$ is $d<ca$, since we require $x_1=ca-d>0$. This sequence satisfies the simple linear recurrence relation $x_{n+1}=ax_n+(a-1)d$, for all $n\geq 1$. Now, for any $n\geq 1$ we define a submonoid $S_n$ of $\mathbb N$ given by 
\begin{equation*}\label{eq-monoid_S_n}
    S_n=\langle\{x_{n+j}:j\in\mathbb N\}\rangle.
\end{equation*}
To simplify notation, given $n\geq1$, we set $\mathbf s_j=x_{n+j}$, for all $j\in \mathbb N$. So, $S_n=\langle\{\mathbf s_{j}:j\in\mathbb N\}\rangle$. Besides, the $\mathbf s_j$'s satisfy the the recurrence relation $\mathbf s_{j+1}=a\mathbf s_j+(a-1)d$, for all $j\geq 0$. 

For our purposes, it will be enough to assume that $d$ is relatively prime to $a$ and $c$ (otherwise, we can divide by an adequate factor to reduce to this case). The following conjectures arise from the results from \cite{gu*, gu, mersenne, repunit, thabit, songthabit, song}:
\begin{enumerate}
    \item The minimal generating set of $S_n$ has the form $\{\mathbf s_j:0\leq j<m\}$ for some $m>0$ (such $m$ must be equal to $e(S_n)$). 
    \item For all $n\geq 1$, $e(S_n)\geq n$.
    \item For all $n\geq1$, ${\rm t}(S_n)\leq e(S_n)-1$.
\end{enumerate}
In this work, we prove all these conjectures in a very general setting, and give a characterization of $e(S_n)$ that involves \textit{repunit numbers} and the elements $\mathbf s_0$ and $\mathbf s_1$.

In general, we have an upper bound for the type of a numerical semigroup $S$: ${\rm t}(S)\leq m(S)-1$, where $m(S)$ is the minimal nonzero element in $S$. We also know that $e(S)\leq m(S)$. Thus, the third conjecture above establishes an improvement to the general upper bound for the type for the numerical semigroups $S_n$. 

One of our main results is the characterization of the of $e(S_n)$. Let $\lambda=\gcd(S_n)$. The $n-$th repunit number in base $a$, $\mathbf R_n$, is defined by 
\begin{equation*}
    \mathbf R_n=1+a+\cdots+a^{n-1}=\frac{a^n-1}{a-1}.
\end{equation*}
We prove that if $m$ is the least positive integer such that $\mathbf s_0/\lambda\leq \mathbf R_m$, and if we also have $\mathbf R_m\leq \mathbf s_1/\lambda$, then $e(S_n)=m$. 

We show that the condition $\mathbf R_m\leq \mathbf s_1/\lambda$ is satisfied in important general cases, for instance, if $d>0$. 

The characterization of the Apéry set ${\rm Ap}(S_n,\mathbf s_0)$ is crucial for finding the Frobenius number, the genus, the pseudo-Frobenius numbers, and the type of $S_n$. To give our characterization of ${\rm Ap}(S_n,\mathbf s_0)$, we need a special set of $r-$tuples of integers, which we describe now. Given an integer $r\geq 1$, a \textit{residual $r-$tuple} (here, we use the terminology given in \cite{mersenne} for such tuples) is a tuple of integers $(\alpha_1,\ldots,\alpha_r)$ such that 
\begin{enumerate}
    \item $0\leq \alpha_i\leq a$ for all $i\in\{1,\ldots, r\}$ and 
    \item if $\alpha_{i_0}=a$ for some $i_0\in \{2,\ldots, r\}$, then $\alpha_i=0$ for all $i\in\{1,\dots,i_0-1\}$.
\end{enumerate}
The set of all residual $r-$tuples will be denoted by $A(r)$. The size of $A(r)$ is $(a^{r+1}-1)/(a-1)$ (see \cite[Theorem 12]{repunit}). An important order relation on $A(r)$ is the \textit{co-lexicographic order $\leq_c$} in which $(\alpha_1, \ldots, \alpha_r)\leq_c(\beta_1,\ldots, \beta_{r})$ if, either, $(\alpha_1, \ldots, \alpha_r)=(\beta_1,\ldots, \beta_{r})$ or there is some $t\in\{1,\ldots, r\}$ such that $\alpha_t<\beta_t$ and $\alpha_j=\beta_j$ for $j=t+1, \ldots, r$. 

If $m=e(S_n)$ and we assume that $\mathbf R_m\leq \mathbf s_1/\lambda$, where $\lambda=\gcd(S_n)$, the we prove that there is a unique residual $(m-1)-$tuple $(\alpha_1, \ldots, \alpha_{m-1})$ such that ${\rm Ap}(S_n, \mathbf s_0)$ is equal to
\begin{equation*}
    \left\{\sum_{j=1}^{m-1}\beta_j\mathbf s_j: (\beta_1,\ldots,\beta_{m-1})\in A(m-1), (\beta_1,\ldots,\beta_{m-1})\leq_c(\alpha_1,\ldots,\alpha_{m-1})\right\}
\end{equation*}
and $\max {\rm Ap}(S_n, \mathbf s_0)=\sum_{j=1}^{m-1}\alpha_j\mathbf s_j$. Moreover, we show that this residual $(m-1)-$tuple $(\alpha_1, \ldots, \alpha_{m-1})$ is the unique that satisfy the relation 
\begin{equation*}
    (\mathbf s_0/\lambda)-1=\sum_{j=1}^{m-1}\alpha_j\mathbf R_j.
\end{equation*}
We test our results on some generating sequences. For instance, we deal with the generating sequence $(2^k-1)2^n-1$, when $k>2^n$ (this case was not treated in \cite{gu}); we also treat sequences of the form $(a^k-1)a^n-1$, for $a\geq 3$ and $k\geq 1$. 

Finally, if $m=e(S_n)$, $\lambda=\gcd(S_n)$ and we assume that $\mathbf R_m\leq \mathbf s_1/\lambda$, then ${\rm t}(S_n)\leq e(S_n)-1$, and verify Wilf's conjecture for the numerical semigroup $(1/\lambda)S_n$, under the given conditions.

\section{Repunit numbers and relations between the generators of \texorpdfstring{$S_n$}{Lg}}\label{sec-relations}

In this small section we show some basic properties of repunit numbers and the generators of $S_n$ that we will use later. For convenience we define $\mathbf R_0=0$. The sequence of repunit numbers in base $a$ satisfies the recurrence relation $\mathbf R_{n+1}=a\mathbf R_n+1$, for all $n\geq 0$. It is easily seen that 
\begin{equation}\label{eq-general_relation_repunits}
    \mathbf R_n=(a-1)\sum_{j=m}^{n-1}\mathbf R_j+\mathbf R_m+n-m,
\end{equation}
whenever $0\leq m<n$. By using the recurrence relation satisfied by the generators of $S_n$, that is, 
$\mathbf s_{j+1}=a\mathbf s_j+(a-1)d$, for all $j\geq 0$, we can derive the following formula that generalizes \eqref{eq-general_relation_repunits}: 
\begin{equation}\label{eq-general_relation_s_j}
 \mathbf s_t=(a-1)\sum_{j=u}^{t-1}\mathbf s_j+\mathbf s_u+(t-u)(a-1)d, 
\end{equation}
where $0\leq u<t$. In the following lemma we give the fundamental relations between the $\mathbf s_j$'s.

The elements of $S_n$ can be represented as linear combinations of the $\mathbf s_j$'s with nonnegative coefficients, but it is convenient to consider linear combinations of the $\mathbf s_j$'s with integer coefficients. 

\begin{proposition}\label{prop-coefficient_s_0_and_s_1}
If $\alpha, \beta_1, \beta_2,\ldots, \beta_{r}$ are integers, then
\begin{equation}\label{eq-coefficient_s_0_and_s_1}
    \alpha \mathbf s_0+\sum_{j=1}^{r}\beta_j \mathbf s_{j}=\left(\alpha-a\sum_{j=2}^{r}\beta_j\mathbf R_{j-1}\right)\mathbf s_0+\left(\sum_{j=1}^{r}\beta_j\mathbf R_j\right)\mathbf s_{1}
\end{equation}
Thus, any linear combination of the $\mathbf s_j$'s can be transformed into a linear combination of $\mathbf s_0$ and $\mathbf s_{1}$. The sum of the coefficients of both sides combinations in \eqref{eq-coefficient_s_0_and_s_1} is the same.
\end{proposition}

\begin{proof}It is enough to prove that for all $j\geq 1$ and integers $\alpha$ and $\beta$,
\begin{equation*}
    \alpha \mathbf s_0+\beta \mathbf s_{j}=\left(\alpha-a\beta\mathbf R_{j-1}\right)\mathbf s_0+\left(\beta\mathbf R_{j}\right)\mathbf s_{1}.
\end{equation*}
By induction on $j$, assume that the result is true for $j\geq 1$ and all integers $\alpha$ and $\beta$. Note that 
$\mathbf s_{j+1}=a\mathbf s_j+(a-1)d=a\mathbf s_j+\mathbf s_1-a\mathbf s_0$. Then we have 
\begin{equation*}
    \alpha\mathbf s_0+\beta\mathbf s_{j+1}=(\alpha-a\beta)\mathbf s_0+\beta\mathbf s_1+a\beta\mathbf s_j.
\end{equation*}
Now, by applying the induction hypothesis to $(\alpha-a\beta)\mathbf s_0+a\beta\mathbf s_j$, we get 
\begin{equation*}
    (\alpha-a\beta)\mathbf s_0+a\beta\mathbf s_j=\left(\alpha-a\beta-a^2\beta\mathbf R_{j-1}\right)\mathbf s_0+\left(a\beta\mathbf R_j\right)\mathbf s_1,
\end{equation*}
and by the recurrence relation between the $\mathbf R_j$'s, we obtain
\begin{align*}
    \alpha \mathbf s_0+\beta \mathbf s_{j+1}&=\left(\alpha-a\beta-a^2\beta\mathbf R_{j-1}\right)\mathbf s_0+\left(\beta+a\beta\mathbf R_j\right)\mathbf s_1\\
    &=\left(\alpha-a\beta\mathbf R_j\right)\mathbf s_0+\left(\beta\mathbf R_{j+1}\right)\mathbf s_{1}.
\end{align*}
This ends the proof of our claim. Now, the sum of coefficients in the right hand side of \eqref{eq-coefficient_s_0_and_s_1} is  
\begin{align*}
    \alpha-a\sum_{j=2}^{r}\beta_j\mathbf R_{j-1}+\sum_{j=1}^{r}\beta_j\mathbf R_j&=\alpha+\beta_1+\sum_{j=2}^{r}\beta_j\left(\mathbf R_j-a\mathbf R_{j-1}\right)\\
    &=\alpha+\beta_1+\sum_{j=2}^{r}\beta_j\\
    &=\alpha+\sum_{j=1}^{r}\beta_j,
\end{align*}
which is the sum of the coefficients in the left hand side of \eqref{eq-coefficient_s_0_and_s_1}.
\end{proof}

\section{\texorpdfstring{The embedding dimension of $S_n$}{Lg}}\label{sec-representation}

We give a characterization of the embedding dimension $S_n$. First, we show that there is a unique positive integer $m$ such that $\{\mathbf s_j: 0\leq j<m\}$ is the minimal generating set of $S_n$, and, of course $e(S_n)=m$. Then, we find conditions on $\mathbf s_0, \mathbf s_1$ and repunit numbers to determine such integer $m$.

\begin{lemma}\label{lm-generatorsetofS_n}
Let $r>0$ be such that $\mathbf s_{r}\in \langle\{\mathbf s_{j}:0\leq j<r\}\rangle$. Then $S_n=\langle\{\mathbf s_{j}:0\leq j<r\}\rangle$. 
\end{lemma}

\begin{proof}
Let $S=\langle\{\mathbf s_{j}:0\leq j<r\}\rangle$. First, we show that $as+(a-1)d\in S$, for all $s\in S\setminus\{0\}$. In fact, if $s\in S\setminus\lbrace 0\rbrace$, then $s=\sum\limits_{i=0}^{r-1}\alpha_i\mathbf s_i$, where $\alpha_i\in \mathbb{N}, i=0, \ldots, r-1$, and some $\alpha_{j}>0$. So 
\begin{align*}
as+(a-1)d&= a\sum_{i=0}^{r-1}\alpha_i \mathbf s_i +(a-1)d \\
&=a\sum_{i=0, i\neq j}^{r-2}\alpha_i \mathbf s_i +a\alpha_j\mathbf s_{j}+(a-1)d\\
&=a\sum_{i=0, i\neq j}^{r-2}\alpha_i \mathbf s_i +a(\alpha_j-1)\mathbf s_{j}+a\mathbf s_j+(a-1)d\\
&=a\sum_{i=0, i\neq j}^{r-2}\alpha_i \mathbf s_i +a(\alpha_j-1)\mathbf s_{j}+\mathbf s_{j+1}.
\end{align*}
If $j<r-1$, then $j+1<r$ and $\mathbf s_{j+1}\in S$. Now, if $j=r-1$, then $j+1=r$ and $\mathbf s_{j+1}=\mathbf s_r\in S$, by hypothesis. Thus, we see that $as+(a-1)d\in S$. 

In order to prove that $S_n=S$, it suffices to prove that $\mathbf s_{j}\in S$ for all $j\geq 0$. Clearly $\mathbf s_0\in S$ and, if $\mathbf s_j\in S$, where $j\geq 0$, then $\mathbf s_{j+1}=a\mathbf s_{j}+(a-1)d\in S$. This ends the proof. 
\end{proof}

Since $S_n$ is finitely generated, there is some $m>0$ such that $\mathbf s_m\in \langle\{\mathbf s_{j}:0\leq j<m\}\rangle$ and $\mathbf s_{m-1}\notin \langle\{\mathbf s_{j}:0\leq j<m-1\}\rangle$. 

\begin{theorem}\label{teo-minimalgeneratingsetS_n}
Assume that for some $m>0$ it holds that $\mathbf s_{m}\in \langle\{\mathbf s_{j}:0\leq j<m\}\rangle$ but $\mathbf s_{m-1}\notin \langle\{\mathbf s_{j}:0\leq j<m-1\}\rangle$. Then $\{\mathbf s_{j}:0\leq j<m\}$ is the minimal generating set of $S_n$. In particular, $m=e(S_n)$.
\end{theorem}

\begin{proof}
By Lemma \ref{lm-generatorsetofS_n}, the set $\{\mathbf s_{j}:0\leq j<m\}$ generates $S_n$. If this is not the minimal generating set of $S_n$, then there exists $j_0\in \{0,\ldots, m-1\}$ such that $\mathbf s_{j_0}\in \langle\{\mathbf s_{j}:0\leq j<j_0\}\rangle$. Note that $j_0<m-1$ since $\mathbf s_{m-1}\notin\langle\{\mathbf s_{j}:0\leq j<m-1\}\rangle$. Again, by Lemma \ref{lm-generatorsetofS_n} we have $S_n=\langle\{\mathbf s_{j}:0\leq j<j_0\}\rangle$, but it follows that $\mathbf s_{m-1}\in \langle\{\mathbf s_{j}:0\leq j<j_0\}\rangle\subseteq \langle\{\mathbf s_{j}:0\leq j<m-1\}\rangle$, which is a contradiction. We conclude that $\{\mathbf s_{j}:0\leq j<m\}$ is the minimal generating set of $S_n$. 
\end{proof}

It is clear that the positive integer $m$ in Theorem \ref{teo-minimalgeneratingsetS_n} is unique, so the embedding dimension of $S_n$ is the unique $m>0$ such that $\mathbf s_{m}\in \langle\{\mathbf s_{j}:0\leq j<m\}\rangle$ and $\mathbf s_{m-1}\notin \langle\{\mathbf s_{j}:0\leq j<m-1\}\rangle$. 

\begin{remark}
In Theorem \ref{teo-minimalgeneratingsetS_n} we only use that $S_n$ is a submonoid of $\mathbb N$; we do not need $S_n$ to be a numerical semigroup yet.
\end{remark}

In the following proposition we characterize the conditions under which the embedding dimension of $S_n$ is $1$. 

\begin{proposition}
If $e(S_n)=1$, then $n=1$ and $ca-d$ divides $c-d$. Conversely, if $ca-d$ divides $c-d$, then $e(S_1)=1$.  
\end{proposition}

\begin{proof}
If $e(S_n)=1$, then $\{\mathbf s_0\}$ generates $S_n$. In particular, $\mathbf s_0\mid \mathbf s_1$. So $\mathbf s_0$ divides $\mathbf s_1-a\mathbf s_0=(a-1)d$, and since $\gcd(\mathbf s_0, d)=1$, we must have $\mathbf s_0\mid a-1$. Therefore, $\mathbf s_0\leq a-1$. Since $d<ac$, we have $ca^n-ca<ca^n-d\leq a-1$, from which $ca(a^{n-1}-1)/(a-1)<1$, and this can only happen when $n=1$. Now, the relation $ca-d=c(a-1)+c-d$ implies that $ca-d$ divides $c-d$, since $\mathbf s_0=ca-d$ divides $a-1$. 

Conversely, if $ca-d$ divides $c-d$, then by the relation $ca-d=c(a-1)+c-d$ and the fact that $ca-d$ and $c$ are relatively prime, we conclude that $ca-d$ divides $a-1$. Now, $ca^2-d=c(a^2-1)+c-d$, from which we see that $\mathbf s_0=ca-d$ divides $\mathbf s_1=ca^2-d$ and by Theorem \ref{teo-minimalgeneratingsetS_n}, $e(S_1)=1$.
\end{proof}

Now, we study the sum of coefficients in representations of $\mathbf s_m$ as linear combinations of $\mathbf s_j$'s, $0\leq j<m$ (if such representations exist). 

\begin{lemma}\label{lm-factor-t}
Assume that for some $m>0$, there are integers $\alpha_0, \alpha_1, \ldots, \alpha_{m-1}$ such that $\mathbf s_m=\alpha_0\mathbf s_0+\alpha_{1}\mathbf s_{1}+\cdots+\alpha_{m-1}\mathbf s_{m-1}$. Then, there exists $t\in\mathbb Z$ such that  
\begin{equation*}
    \sum_{j=0}^{m-1}\alpha_j=1+tca^n.
\end{equation*}
Moreover, if the $\alpha_j$'s are nonnegative, then $t>0$. 
\end{lemma}

\begin{proof}
The equation $\mathbf s_m=\sum_{j=0}^{m-1}\alpha_j\mathbf s_j$ implies that 
\begin{equation*}
    ca^{n+m}-d=\sum_{j=0}^{m-1}\alpha_jca^{n+j}-d\sum_{j=0}^{m-1}\alpha_j,
\end{equation*}
from which $d\left(\sum_{j=0}^{m-1}\alpha_j-1\right)=sca^n$, where $s=\sum_{j=0}^{m-1}\alpha_ja^{j}-a^{m}$. Since $d$ is relatively prime to $a$ and $c$, we obtain that $d$ divides $s$. So, $\sum_{j=0}^{m-1}\alpha_j=1+tca^n$, where $t=s/d$. Clearly, if the $\alpha_j$'s are nonnegative, it must be $t>0$. 
\end{proof}

\begin{theorem}\label{teo-lower_bound_of_e(S_n)}
For all $n\in\mathbb Z^+$, we have $e(S_n)\geq n$.
\end{theorem}

\begin{proof}
It suffices to show that $\mathbf s_{n-1}\notin \langle\{\mathbf s_j:0\leq j<n-1\}\rangle$. If we had $\mathbf s_{n-1}=\sum_{j=0}^{n-2}\alpha_j\mathbf s_j$, for some nonnegative $\alpha_j$'s, then by Lemma \ref{lm-factor-t}, there is some $t>0$ such that $\sum_{j=0}^{n-2}\alpha_j=1+tca^n$. Hence,
\begin{equation*}
    \mathbf s_{n-1}=\sum_{j=0}^{n-2}\alpha_j\mathbf s_j\geq \left(\sum_{j=0}^{n-2}\alpha_j\right)\mathbf s_0=(1+tca^n)\mathbf s_0\geq (1+ca^n)\mathbf s_0.
\end{equation*}
Now, since $d<ca$, we have $d\leq ca-1\leq a^{n-1}(ca-1)$, from which it follows that $\mathbf s_0\geq a^{n-1}$, which is equivalent to $(1+ca^n)\mathbf s_0>\mathbf s_{n-1}$. This gives us a contradiction. 
\end{proof}

Theorem \ref{teo-lower_bound_of_e(S_n)} is precisely our conjecture about the lower bound on the embedding dimension of $e(S_n)$.   

Now we go to our main result about the embedding dimension of $S_n$. After its proof, we show some consequences and examples. We need the following lemma, which is easy to verify. 

\begin{lemma}\label{lm-s_m}
Let $t$ be a positive integer. If $\lambda$ divides $\mathbf s_0$ and $\mathbf s_1$, then 
\begin{equation*}
    \mathbf s_t=\left(\frac{\mathbf s_1}{\lambda}-\mathbf R_t+1\right)\mathbf s_0+\left(\mathbf R_t-\frac{\mathbf s_0}{\lambda}\right)\mathbf s_1.
\end{equation*}
\end{lemma}

Before the next theorem, we prove that $\gcd(S_n)=\gcd(\mathbf s_0, \mathbf s_1)=\gcd(\mathbf s_0, a-1)$. In fact, since $d$ is relatively prime to $a$ and $c$, we have $\mathbf s_0=ca^n-d$ is relatively prime to $a, c$ and $d$. The relation $\mathbf s_1=a\mathbf s_0+(a-1)d$ implies that $\gcd(\mathbf s_0, \mathbf s_1)=\gcd(\mathbf s_0, (a-1)d)=\gcd(\mathbf s_0,a-1)$, where, for the last equality, we use the fact that $\gcd(\mathbf s_0, d)=1$. Now, using the relation $\mathbf s_j=a\mathbf s_{j-1}+(a-1)d$, $j=1,2,\ldots,$ we find that every common divisor of $\mathbf s_0$ and $a-1$ divides $\mathbf s_j$ for all $j\geq 0$. Thus, $\gcd(\mathbf s_0, a-1)$ divides $\gcd(S_n)$. Clearly, $\gcd(S_n)$ divides $\gcd(\mathbf s_0, \mathbf s_1)$, and this ends the proof of our claim. 

\begin{theorem}\label{theo-embedding_dimension}
Let $\lambda=\gcd(S_n)$. Let $m$ be the smallest positive integer such that $\mathbf s_0/\lambda\leq \mathbf R_m$ and assume that $\mathbf s_1/\lambda\geq \mathbf R_m$. Then, $e(S_n)=m$.
\end{theorem}

\begin{proof}
By Theorem \ref{teo-minimalgeneratingsetS_n} we must show that $\mathbf s_m$ can be written as a linear combination of $\mathbf s_0, \ldots, \mathbf s_{m-1}$, with nonnegative integer coefficients, and also that $\mathbf s_{m-1}$ cannot be written as linear combination of $\mathbf s_0,\ldots, \mathbf s_{m-1}$ with nonnegative integer coefficients. In fact, by Lemma \ref{lm-s_m} we have
\begin{equation*}
    \mathbf s_m=\left(\frac{\mathbf s_1}{\lambda}-\mathbf R_m+1\right)\mathbf s_0+\left(\mathbf R_m-\frac{\mathbf s_0}{\lambda}\right)\mathbf s_1, 
\end{equation*}
and by hypothesis, the coefficients in this representation are nonnegative integers. Let $A_0=\left(\mathbf s_1/\lambda\right)-\mathbf R_m+1$ and $B_0=\mathbf R_m-\left(\mathbf s_0/\lambda\right)$, so that $\mathbf s_m=A_0\mathbf s_0+B_0\mathbf s_1$.

Now we have to show that $\mathbf s_{m-1}$ cannot be written as linear combination of $\mathbf s_0,\ldots, \mathbf s_{m-1}$ with nonnegative integer coefficients. We make the following claim: \textit{ For any representation $\mathbf s_m=\beta_0\mathbf s_0+\cdots+\beta_{m-1}\mathbf s_{m-1}$, where $\beta_1, \ldots, \beta_{m-1}$ are nonnegative integers, there is a nonnegative integer $u$ such that 
\begin{equation*}
    \sum_{j=0}^{m-1}\beta_j=1+\left(u+1\right)\frac{(a-1)}{\lambda}ca^n.
\end{equation*}}
To prove the claim, note that by Proposition \ref{prop-coefficient_s_0_and_s_1}, if $\mathbf s_{m}=\beta_0\mathbf s_0+\cdots+\beta_{m-1}\mathbf s_{m-1}$, where the $\beta_j$'s are nonnegative, then we can derive a representation $\mathbf s_m=A\mathbf s_0+B\mathbf s_1$ for some integers $A, B$, and $A+B=\sum_{j=0}^{m-1}\beta_j$. 
By Lemma \ref{lm-factor-t}, there is an integer $t>0$ such that $\sum_{j=0}^{m-1}\beta_j=1+tca^n$, so $A+B=1+tca^n$. Therefore, we have $A\mathbf s_0+B\mathbf s_1=\mathbf s_m=A_0\mathbf s_0+B_0 \mathbf s_1$, so there is an integer $u$ such that $A-A_0=u\mathbf s_1/\lambda$ and $B-B_0=-u\mathbf s_0/\lambda$. Hence, 
\begin{equation*}
    (A+B)-(A_0+B_0)=u\left(\mathbf s_1/\lambda-\mathbf s_0/\lambda\right)=uc[(a-1)/\lambda]a^n.
\end{equation*}
It follows that  
\begin{equation*}
 A+B=uc[(a-1)/\lambda]a^n+(A_0+B_0)=1+(u+1)[(a-1)/\lambda]ca^n, 
\end{equation*}
which implies that $t=(u+1)(a-1)/\lambda$. Since $t\geq 1$, it must be $u\geq 0$. This ends the proof of the claim. 

To continue the proof of the theorem, by contradiction, we assume that there is a representation $\mathbf s_{m-1}=\sum_{j=0}^{m-2}b_j\mathbf s_j$, where the $b_j$'s are nonnegative, and $b_{j_0}>0$ for some $j_0\in\{0,\ldots, m-2\}$. Hence, 
\begin{align*}
    \mathbf s_{m}&=a\mathbf s_{m-1}+(a-1)d\\
    &=\sum_{j=0}^{m-2}(ab_j)\mathbf s_j+(a-1)d\\
    &=\sum_{j=0, j\neq j_0}^{m-2}(ab_j)\mathbf s_j+ab_{j_0}\mathbf s_{j_0}+(a-1)d\\
    &=\sum_{j=0, j\neq j_0}^{m-2}(ab_j)\mathbf s_j+a(b_{j_0}-1)\mathbf s_{j_0}+\mathbf s_{j_0+1}.
\end{align*}
The sum of coefficients in this representation of $\mathbf s_{m}$ is 
\begin{equation*}
    \sum_{j=0, j\neq j_0}^{m-2}(ab_j)+a(b_{j_0}-1)+1=a\sum_{j=0}^{m-2}b_j-a+1.
\end{equation*}
By the claim there exists $u_0\geq 0$ such that $a\sum_{j=0}^{m-2}b_j-a+1=1+(u_0+1)[(a-1)/\lambda]ca^n$. Then, $\sum_{j=0}^{m-2}b_j=1+(u_0+1)[(a-1)/\lambda]ca^{n-1}$. On the other hand, by Lemma \ref{lm-factor-t} we have $\sum_{j=0}^{m-2}b_j=1+t_1ca^n$ for some $t_1>0$, so that $1+(u_0+1)[(a-1)/\lambda]ca^{n-1}=1+t_1ca^n$. It follows that $(u_0+1)[(a-1)/\lambda]=t_1a$; but, since $a$ and $(a-1)/\lambda$ are relatively prime, $a$ must divide $u_0+1$. Then, $u_0+1\geq a$ and $\sum_{j=0}^{m-2}b_j=1+(u_0+1)[(a-1)/\lambda]ca^{n-1}\geq 1+[(a-1)/\lambda]ca^n$.
Therefore, 
\begin{equation*}
    \mathbf s_{m-1}=\sum_{j=0}^{m-2}b_j\mathbf s_j \geq \left(\sum_{j=0}^{m-2}b_j\right)\mathbf s_0=(1+[(a-1)/\lambda]ca^n)\mathbf s_0.
\end{equation*}
Now, by hypothesis we have $\mathbf s_0>\lambda\mathbf R_{m-1}$, which is equivalent to $[(a-1)/\lambda]\mathbf s_0\geq a^{m-1}$, and this is equivalent to $\mathbf s_{m-1}<(1+[(a-1)/\lambda]ca^n)\mathbf s_0$, so we have a contradiction. This ends the proof of Theorem \ref{theo-embedding_dimension}.
\end{proof}

\begin{example}
Let $a\geq 2$ and consider the generating sequence $x_n=a^n-1$. Here, $\lambda=a-1$, and the minimal positive integer $m$ such that $\mathbf s_0/\lambda=(a^n-1)/(a-1)\leq \mathbf R_m=(a^m-1)/(a-1)$, is precisely $m=n$. Clearly, $\mathbf R_n<\mathbf s_1/\lambda=(a^{n+1}-1)/(a-1)$. By Theorem \ref{theo-embedding_dimension}, $e(S_n)=n$.
\end{example}

The condition $\mathbf R_m\leq \mathbf s_1/\lambda$ in Theorem \ref{theo-embedding_dimension} is satisfied in many important cases, as in the following proposition.   

\begin{proposition}\label{prop-case_d>0}
Suppose that $d>0$ and let $\lambda=\gcd(S_n)$. If $m$ is the smallest positive integer such that $\mathbf s_0/\lambda\leq \mathbf R_m$, then $e(S_n)=m$.
\end{proposition}

\begin{proof}
We have to prove that $\mathbf R_m\leq \mathbf s_1/\lambda$. Suppose, on the contrary, that $\lambda\mathbf R_m>\mathbf s_1$. Then, 
\begin{equation*}
    \lambda a^{m-1}=\lambda\frac{a^m-a^{m-1}}{a-1}=\lambda(\mathbf R_m-\mathbf R_{m-1})>\mathbf s_1-\mathbf s_0=(a-1)ca^n.
\end{equation*}
It follows that $a^{m-1}>\frac{a-1}{\lambda}ca^n$, so $\lambda\mathbf R_{m-1}\geq ca^n>ca^n-d=\mathbf s_0$, which is a contradiction.
\end{proof}

\begin{corollary}\label{cor-case_a=2}
Consider a generating sequence of the form $x_n=c2^n-d$, where $c>1$. If $0<d\leq 2^n$ and $k$ is the smallest positive integer such that $c\leq 2^k$, then $e(S_n)=n+k$.
\end{corollary}

\begin{proof}
Here we have $a=2$, so $\lambda=1$, since $\lambda$ divides $a-1$. Write $c=2^k-r$, where $0\leq r<2^{k-1}$. Note that $\mathbf s_0=2^{n+k}-(r+d)\leq 2^{n+k}-1=\mathbf R_{n+k}$. Now, we must show that $\mathbf R_{n+k-1}<\mathbf s_0$. We have $c\geq 2^{k-1}+1$, so 
\begin{equation*}
    \mathbf s_0=c2^n-d\geq 2^{n+k-1}+2^n-d\geq 2^{n+k-1}>\mathbf R_{n+k-1}.
\end{equation*}
This ends the proof. 
\end{proof}

In a similar way we can prove the following. 
\begin{corollary}\label{cor-the_case_a=2}
Consider a generating sequence of the form $x_n=c2^n-1$. If $k$ is the smallest nonnegative integer such that $c\leq 2^k$, then $e(S_n)=n+k$.
\end{corollary}

\begin{example}\label{ex-embedding_dimension_a=2}
\begin{enumerate}
\item Fix $k\geq 2$ and consider the generating sequence $x_n=(2^k-1)2^n-1$. By Corollary \ref{cor-the_case_a=2}, we have $e(S_n)=n+k$ (see \cite{gu}).

\item For $k\geq 1$, consider the generating sequence $x_n=(2^k+1)2^n-(2^k-1)$. In this case, the smallest power of $2$ greater than or equal to $c=2^k+1$ is $2^{k+1}$, and $d=2^k-1$. So, we have $d\leq 2^n$ if and only if $k\leq n$. Thus, by Corollary \ref{cor-case_a=2}, if $k\leq n$, then $e(S_n)=n+k+1$ (see \cite{songthabit}).

\item Let $k\geq 1$ and $x_n=(2^k+1)2^n-(2^k-1)$, where $k>n$. Here, we have $c=2^k+1$. In this case we have $\mathbf R_{n+k-1}<\mathbf s_0\leq \mathbf R_{n+k}$, so by Corollary \ref{cor-case_a=2}, $e(S_n)=n+k$ (see \cite{songthabit}).

\item Let $a\geq 3$, $k\geq 1$, and consider the generating sequence $x_n=(a^k-1)a^n-1$. It is easy to see that, for all $n\geq1$, $\lambda=\gcd(S_n)=1$. We claim that $e(S_n)=n+k+1$. To show this, by Proposition \ref{prop-case_d>0}, we must prove that $\mathbf R_{n+k}<\mathbf s_0\leq \mathbf R_{n+k+1}$. In fact, observe first that 
\begin{equation*}
    \mathbf s_0=a^{n+k}-a^n-1<a^{n+k}+a^{n+k-1}+\cdots+a+1=\mathbf R_{n+k+1}.
\end{equation*}
On the other hand, it is easily seen that $(a-1)\mathbf s_0\geq a^{n+k}$, from which it follows that $\mathbf s_0>(a^{n+k}-1)/(a-1)=\mathbf R_{n+k}$. 
\end{enumerate}
\end{example}

Now we return to the condition $\mathbf s_1/\lambda\geq \mathbf R_{m}$ in Theorem \ref{theo-embedding_dimension}. When $d<0$, the conclusion of Theorem \ref{theo-embedding_dimension} may not be true, for it may happen that $\mathbf R_m>\mathbf s_1/\lambda$. A counterexample can be found by taking $a=2$, $c=1$, $d=-33$, and $n=1$. In this case $m=6$ but $e(S_1)=7$. In the following proposition we give conditions under which $\mathbf R_m\leq \mathbf s_1/\lambda$ is satisfied, although $d<0$. 

\begin{proposition}
Assume $d<0$. Let $k$ and $l$ be the only positive integers such that $\mathbf R_{k-1}< -d\leq \mathbf R_k$ and $\mathbf R_{l-1}<c\leq \mathbf R_l$. If $n\geq k$ and $\gcd(S_n)=1$, then $\mathbf s_0\leq \mathbf R_{n+l}\leq \mathbf s_1$. Moreover, $e(S_n)=n+l$. 
\end{proposition}

\begin{proof}
There are positive integers $d_0$ and $c_0$ such that $
-d=\mathbf R_{k-1}+d_0$ and $c=\mathbf R_{l-1}+c_0$, where $d_0\leq a^{k-1}$ and $c_0\leq a^{l-1}$. Hence, 
\begin{align*}
    \mathbf s_0=ca^n+(-d)=c_0a^n+\sum_{j=0}^{l-2}a^{n+j}+d_0+\sum_{j=0}^{k-2}a^j
    \leq a^{n+l-1}+\sum_{j=0}^{l-2}a^{n+j}+a^{k-1}+\sum_{j=0}^{k-2}a^j.
\end{align*}
If $n\geq k$, then it follows that $\mathbf s_0\leq \mathbf R_{n+l}$. Now, 
\begin{align*}
    \mathbf s_1&=c_0a^{n+1}+\sum_{j=0}^{l-2}a^{n+1+j}+d_0+\sum_{j=0}^{k-2}a^j\\
    &=c_0a^{n+1}+\mathbf R_{n+l}+d_0-a^{k-1}(a^{n-k+1}+\cdots+1)\\
    &=c_0a^{n+1}+\mathbf R_{n+l}+d_0-a^{k-1}\mathbf R_{n-k+2},
\end{align*}
and it is easy to verify that 
\begin{equation*}
    c_0a^{n+1}+d_0\geq a^{n+1}+1\geq a^{k-1}\mathbf R_{n-k+2},
\end{equation*}
so $\mathbf s_1\geq \mathbf R_{n+l}$. It is not difficult to see that $\mathbf s_0>\mathbf R_{n+l-1}$, so we have $e(S_n)=n+l$, by Theorem \ref{theo-embedding_dimension}.
\end{proof}

\section{Residual tuples and the Apéry set \texorpdfstring{${\rm Ap}(S_n,\mathbf s_0)$}{}}

In this section we describe a way of finding the maximum element of ${\rm Ap}(S_n, \mathbf s_0)$, and we show a characterization of the elements of ${\rm Ap}(S_n, \mathbf s_0)$. Throughout this section, we assume the embedding dimension of $S_n$ is at least 2. We start with the following proposition, which is a common result found in \cite{gu*, gu, mersenne, repunit, thabit, songthabit, song}.  Recall that $A(m-1)$ stands for the set of residual $(m-1)-$tuples.  

\begin{proposition}\label{prop-residual}
Let $n$ be a positive integer and $m=e(S_n)$. If $x\in {\rm Ap}(S_n, \mathbf s_0)$, then there exists $(\alpha_{1}, \ldots, \alpha_{m-1})\in A(m-1)$ such that $x=\alpha_{1}\mathbf s_{1}+\cdots+\alpha_{m-1}\mathbf s_{m-1}$.
\end{proposition}

\begin{proof}
We proceed by induction on $x$. For $x=0$ we have the residual $(m-1)-$tuple $(0,\ldots,0)$ such that $x=0\mathbf s_{1}+\cdots+0\mathbf s_{m-1}$. Suppose that $x>0$ and let $j=\min \{i \in \{0,\ldots, m-1\}: x-\mathbf s_{i}\in S_n\}$. Observe that $j\neq 0$ since $x\in {\rm Ap}(S_n, \mathbf s_0)$. By induction, there exists $(\alpha_1, \ldots, \alpha_{m-1})\in A(m-1)$ such that $x-\mathbf s_{j}=\alpha_1\mathbf s_{1}+\cdots+\alpha_{m-1}\mathbf s_{m-1}$. Hence, $x=\alpha_1\mathbf s_{1}+\cdots+(\alpha_j+1)\mathbf s_{j}+\cdots+\alpha_{m-1}\mathbf s_{m-1}$. To conclude the proof, we only need to show that $(\alpha_1, \ldots, \alpha_j+1,\ldots,\alpha_{m-1})\in A(m-1)$. If $\alpha_j+1=a+1$, then we get $(\alpha_j+1)\mathbf s_{j}=(a+1)\mathbf s_{j}=\mathbf s_{j}+a\mathbf s_{j}=a\mathbf s_{j-1}+\mathbf s_{j+1}$. Since $\mathbf s_{j+1}\in S_n$, we see that $x-\mathbf s_{j-1}\in S_n$, contradicting the minimality of $j$. If there exists $k>j$ such that $\alpha_k=a$, then we obtain that $\mathbf s_{j}+a\mathbf s_{k}=a\mathbf s_{j-1}+\mathbf s_{k+1}$, so $x-\mathbf s_{j-1}\in S_n$, which contradicts the minimality of $j$. Also, by the minimality of $j$ we have $\alpha_1=\cdots=\alpha_{j-1}=0$, and consequently, $(\alpha_1, \ldots \alpha_j+1,\ldots,\alpha_{m-1})\in A(m-1)$.
\end{proof}

\begin{lemma}\label{lm-s_r+1>combination_residual_tuple}
Let $t$ be a positive integer. Then $\mathbf R_{t+1}>\sum_{j=1}^{t}\alpha_j\mathbf R_j$ for all residual $t-$tuples $(\alpha_1, \ldots, \alpha_t)$. 
\end{lemma}

\begin{proof}
Let $(\alpha_1,\ldots, \alpha_{t})$ be a residual $t-$tuple. If $\alpha_{j_0}=a$ for some $j_0\in\{1,\ldots, t\}$, then 
\begin{align*}
   \sum_{j=1}^{t}\alpha_j\mathbf R_j&\leq \mathbf R_{j_0}+(a-1)\sum_{j=j_0}^{t}\mathbf R_j\\
   &\leq \mathbf R_{j_0}+\left(\sum_{j=j_0}^{t}a^{j}-(t+1-j_0)\right)\\
   &=\frac{a^{j_0}-1}{a-1}+\left(a^{j_0}\mathbf R_{t+1-j_0}-(t+1-j_0)\right)\\
   &=\frac{1}{a-1}\left(a^{j_0}-1+a^{j_0}\left(a^{t+1-j_0}-1\right)-(a-1)(t+1-j_0)\right)\\
   &=\frac{1}{a-1}\left(a^{t+1}-1-(a-1)(t+1-j_0)\right)\\
   &=\mathbf R_{t+1}-(t+1-j_0)\\
   &<\mathbf R_{t+1}
\end{align*}
In a similar way it can be shown that if $\alpha_j<a$ for all $j$, then $\sum_{j=1}^{t}\alpha_j\mathbf R_j<\mathbf R_{t+1}$. 
\end{proof}

Now, we prove some properties of residual tuples. Recall the colexicographic order $\leq_c$ on the set of residual $r-$tuples $A(r)$: $(\alpha_1, \ldots, \alpha_r)\leq_c(\beta_1, \ldots, \beta_r)$ if and only if $(\alpha_1, \ldots, \alpha_r)=(\beta_1, \ldots, \beta_r)$ or there exists $t\in \{1,\ldots, t\}$ such that $\alpha_t<\beta_t$ and $\alpha_i=\beta_i$ for all $i>t$. 

\begin{proposition}
If $(\alpha_1, \ldots, \alpha_{r})$ and $(\beta_1,\ldots, \beta_{r})$ are residual $r-$ tuples, then $(\alpha_1, \ldots, \alpha_{r})<_c(\beta_1,\ldots, \beta_{r})$ if and only if $\sum_{j=1}^{r}\alpha_j\mathbf R_j< \sum_{j=1}^{r}\beta_j\mathbf R_j$.
\end{proposition}

\begin{proof}
If $(\alpha_1, \ldots, \alpha_{r})<_c(\beta_1,\ldots, \beta_{r})$, then there is some $t\in \{1,\ldots, r\}$ such that $\alpha_t<\beta_t$ and $\alpha_i=\beta_i$ for $i=t+1, \ldots, r$. Thus,  
\begin{align*}
\sum_{j=1}^{r}\alpha_j\mathbf R_j&=\sum_{j=1}^{t-1}\alpha_j\mathbf R_j+\alpha_t\mathbf R_t+\sum_{j=t+1}^{r}\alpha_j\mathbf R_j\\
&<\mathbf R_t+\alpha_t\mathbf R_t+\sum_{j=t+1}^{r}\alpha_j\mathbf R_j\\
&\leq \beta_t\mathbf R_t+\sum_{j=t+1}^{r}\beta_j\mathbf R_j\\
&\leq \sum_{j=1}^{r}\beta_j\mathbf R_j
\end{align*}
In the second line above we applied Lemma \ref{lm-s_r+1>combination_residual_tuple}. For the converse, suppose that $\sum_{j=1}^{r}\alpha_j\mathbf R_j< \sum_{j=1}^{r}\beta_j\mathbf R_j$. Clearly, it cannot be $(\alpha_1, \ldots, \alpha_{r})=(\beta_1,\ldots, \beta_{r})$. Now, if $(\alpha_1, \ldots, \alpha_{r})>(\beta_1,\ldots, \beta_{r})$, then it would follow that $\sum_{j=1}^{r}\alpha_j\mathbf R_j>\sum_{j=1}^{r}\beta_j\mathbf R_j$, which is a contradiction. This ends the proof. 
\end{proof}

\begin{corollary}
For all $r\geq 1$, 
\begin{equation*}
    \left\{\sum_{j=1}^{r}\alpha_j\mathbf R_j:(\alpha_1,\ldots, \alpha_r)\in A(r)\right\}=\{0,1,\ldots, \mathbf R_{r+1}-1\}.
\end{equation*}
Moreover, there is an isomorphism of partially ordered sets $\varphi:A(r)\to \{0,\ldots, \mathbf R_{r+1}-1\}$ given by 
\begin{equation*}
    \varphi(\alpha_1, \ldots, \alpha_r)=\sum_{j=1}^r\alpha_j\mathbf R_j,
\end{equation*}
for all residual $r-$tuple $(\alpha_1, \ldots, \alpha_r)$. 
\end{corollary}

\begin{proof}
The minimum and maximum residual $r-$tuples with respect to the co-lexicographic order are $(0,\ldots, 0)$ and $(0,\ldots, 0,a)$, respectively; these $r-$ tuples correspond to the integers $0$ and $a\mathbf R_r=\mathbf R_{r+1}-1$. So, 
\begin{equation*}
\left\{\sum_{j=1}^{r}\alpha_j\mathbf R_j:(\alpha_1,\ldots, \alpha_r)\in A(r)\right\}\subseteq\{0,\ldots, \mathbf R_{r+1}-1\}.    
\end{equation*}
Since both sets have $\mathbf R_{r+1}$ elements, they are equal. 
\end{proof}

\begin{corollary}\label{cor-representation_residual}
Let $r\geq 1$. For each integer $t$ such that $0\leq t<\mathbf R_{r+1}$, there exists a unique residual $r-$tuple $(\alpha_1, \ldots, \alpha_r)$ such that $t=\sum_{j=1}^{r}\alpha_j\mathbf R_j$. 
\end{corollary}

Let $(\alpha_1, \ldots, \alpha_r)\neq (0,\ldots, 0,a)$ be a residual $r-$tuple. We want to find the immediate successor of $(\alpha_1, \ldots, \alpha_r)$ in $A(r)$. To do this, we just have to represent $1+\sum_{j=1}^r\alpha_j\mathbf R_j$ in the form $\sum_{j=1}^r\beta_j\mathbf R_j$, and so $(\beta_1, \ldots, \beta_r)$ is the immediate successor of $(\alpha_1,\ldots, \alpha_r)$. We have two cases. If $\alpha_i=a$ for some $i<r$, then  
\begin{align*}
    1+\sum_{j=1}^r\alpha_j\mathbf R_j&=1+a\mathbf R_i+\sum_{j=i+1}^r\alpha_j\mathbf R_j\\
    &=\mathbf R_{i+1}+\sum_{j=i+1}^r\alpha_j\mathbf R_j\\
    &=(1+\alpha_{i+1})\mathbf R_{i+1}+\sum_{j=i+2}^r\alpha_j\mathbf R_j. 
\end{align*}
Therefore, the immediate successor of $(0,\ldots,0,a, \alpha_{i+1},\ldots, \alpha_r)$ is $(0,\ldots,0,0,1+\alpha_{i+1},\ldots, \alpha_r)$. Similarly, if $\alpha_j<a$ for all $j$, then the immediate successor of $(\alpha_1,\ldots, \alpha_r)$ is $(1+\alpha_1,\ldots, \alpha_r)$. 

\begin{lemma}\label{lm-size_residual_tuples_less_than_(alpha_j)}
Let $(\alpha_1, \ldots, \alpha_r)$ be a residual $r-$tuple. Then, the number of residual $r-$tuples less than or equal to $(\alpha_1, \ldots, \alpha_r)$ with respect to the colexicographic order is 
\begin{equation*}
    1+\sum_{j=1}^r\alpha_j\mathbf R_j.
\end{equation*}
\end{lemma}

\begin{proof}
Let $B_r$ be the set of $r-$tuples less than or equal to $(\alpha_1, \ldots, \alpha_r)$ with respect to the co-lexicographic order. If $r=1$, then $|B_1|=1+\alpha_1=1+\alpha_1\mathbf R_1$. Assume the result is true for $r\geq 1$. Let $(\alpha_1, \ldots, \alpha_{r+1})$ be a residual $(r+1)-$tuple. The set $B_{r+1}$ is the disjoint union of subsets $C_1$ and $C_2$, where $C_1$ is the set of all residual $(r+1)-$tuples $(\beta_1, \ldots, \beta_{r+1})$ such that $\beta_{r+1}<\alpha_{r+1}$ and $C_2$ is the set of all residual $(r+1)-$tuples $(\beta_1, \ldots, \beta_{r+1})$ such that $\beta_{r+1}=\alpha_{r+1}$. Each element in $C_1$ can be represented uniquely in the form $(\beta_1, \ldots, \beta_{r+1})$, where $0\leq \beta_{r+1}<\alpha_{r+1}$ and $(\beta_1, \ldots, \beta_{r})$ is a residual $r-$tuple, so the set $C_1$ has $\alpha_{r+1}\mathbf R_{r+1}$ elements.

Let us count the number of elements in $C_2$. If $\alpha_{r+1}=a$, then $C_2$ only contains the $(r+1)-$tuple, $(0,\ldots, 0, a)$. In this case, the set $B_r$ only contains the $r-$tuple, $(0,\ldots, 0)$. Now, if $\alpha_{r+1}<a$, then the elements in $C_2$ are representable uniquely in the form $(\beta_1, \ldots, \beta_r, \alpha_{r+1})$, where $(\beta_1, \ldots, \beta_r)\in B_r$. In any case, $C_2$ has the same number of elements as $B_r$. This number is $1+\sum_{j=1}^r\alpha_j\mathbf R_j$, by the induction hypothesis. Thus, $B_{r+1}$ has $1+\sum_{j=1}^{r+1}\alpha_j\mathbf R_j$ elements.
\end{proof}

Now, we give our characterization of the Apéry set ${\rm Ap}(S_n, \mathbf s_0)$ in terms of residual $r-$tuples. By Proposition \ref{prop-residual}, if $m=e(S_n)$, then $\mathbf s_0=|{\rm Ap}(S_n,\mathbf s_0)|\leq \mathbf R_m$. So, if $\lambda=\gcd(S_n)$, then $\mathbf s_0/\lambda\leq \mathbf R_m$. Thus, $0\leq\mathbf s_0/\lambda-1<\mathbf R_m$ and, by Corollary \ref{cor-representation_residual}, there is a unique residual $(m-1)-$tuple $(\alpha_1, \ldots, \alpha_{m-1})$ such that $\mathbf s_0/\lambda=1+\sum_{j=1}^{m-1}\alpha_j\mathbf R_j$.  

\begin{theorem}\label{teo-Apery_set_characterization}
Let $m=e(S_n)$, $\lambda=\gcd(S_n)$, $(\alpha_1, \ldots, \alpha_{m-1})$ be the unique residual $(m-1)-$tuple such that $\mathbf s_0/\lambda=1+\sum_{j=1}^{m-1}\alpha_j\mathbf R_j$, and assume that $\mathbf R_m\leq \mathbf s_1/\lambda$. Then, ${\rm Ap}(S_n, \mathbf s_0)$ coincides with the following set
\begin{equation}\label{eq-apery_residual}
\left\{\sum_{j=1}^{m-1}\beta_j\mathbf s_j: (\beta_1,\ldots,\beta_{m-1})\in A(m-1), (\beta_1,\ldots,\beta_{m-1})\leq_c(\alpha_1,\ldots,\alpha_{m-1})\right\}.
\end{equation}
\end{theorem}

\begin{proof}
Let us denote by $T(m)$ the set in \eqref{eq-apery_residual}. The size of $T(m)$ is precisely $1+\sum_{j=1}^{m-1}\alpha_j\mathbf R_{j}=\mathbf s_0/\lambda$. For any $(m-1)-$tuple $(\beta_1, \ldots, \beta_{m-1})$, where $\beta_1, \ldots, \beta_{m-1}$ are integers, we have 
\begin{equation*}
\sum_{j=1}^{m-1}\beta_j\mathbf s_j\equiv (a-1)d\sum_{j=1}^{m-1}\beta_j\mathbf R_j\ ({\rm mod}\ \mathbf s_0). 
\end{equation*}
Since $\lambda=\gcd(S_n)=\gcd(\mathbf s_0, a-1)$, we get
\begin{equation*}
\sum_{j=1}^{m-1}\beta_j\left(\mathbf s_j/\lambda\right)\equiv ((a-1)/\lambda)d\sum_{j=1}^{m-1}\beta_j\mathbf R_j\ ({\rm mod}\ \mathbf s_0/\lambda).
\end{equation*}
Recall that $(1/\lambda)S_n$ is a numerical semigroup and the Apéry set of $(1/\lambda)\mathbf s_0$ is given by ${\rm Ap}((1/\lambda)S_n, (1/\lambda)\mathbf s_0)=(1/\lambda){\rm Ap}(S_n, \mathbf s_0)$. 

Now, when $(\beta_1, \ldots, \beta_{m-1})$ runs over the set of residual $(m-1)-$tuples less than or equal to $(\alpha_1, \ldots, \alpha_{m-1})$ with respect to the colexicographic order, we obtain $\mathbf s_0/\lambda$ integers of the form $\sum_{j=1}^{m-1}\beta_j(\mathbf s_j/\lambda)$ that are mutually incongruent modulo $\mathbf s_0/\lambda$. In fact, if $(\beta_1, \ldots, \beta_{m-1})$ and $(\gamma_1, \ldots, \gamma_{m-1})$ are less than or equal to $(\alpha_1, \ldots, \alpha_{m-1})$ and $\sum_{j=1}^{m-1}\beta_j(\mathbf s_j/\lambda)\equiv\sum_{j=1}^{m-1}\gamma_j(\mathbf s_j/\lambda)\ ({\rm mod}\ \mathbf s_0/\lambda)$, then 
\begin{equation*}
((a-1)/\lambda)d\sum_{j=1}^{m-1}\beta_j\mathbf R_j\equiv ((a-1)/\lambda)d\sum_{j=1}^{m-1}\gamma_j\mathbf R_j\ ({\rm mod}\ \mathbf s_0/\lambda).
\end{equation*}
Since $\mathbf s_0/\lambda$ is relatively prime to $(a-1)/\lambda$ and $d$, we obtain 
\begin{equation*}
\sum_{j=1}^{m-1}\beta_j\mathbf R_j\equiv \sum_{j=1}^{m-1}\gamma_j\mathbf R_j\ ({\rm mod}\ \mathbf s_0/\lambda).
\end{equation*}
Now, both $\sum_{j=1}^{m-1}\beta_j\mathbf R_j$ and $\sum_{j=1}^{m-1}\gamma_j\mathbf R_j$ are less than or equal to $\sum_{j=1}^{m-1}\alpha_j\mathbf R_j=\mathbf s_0/\lambda-1$, so $\sum_{j=1}^{m-1}\beta_j\mathbf R_j=\sum_{j=1}^{m-1}\gamma_j\mathbf R_j$. Hence, $(\beta_1, \ldots, \beta_{m-1})=(\gamma_1, \ldots, \gamma_{m-1})$. We have shown that the set of integers $\sum_{j=1}^{m-1}\beta_j(\mathbf s_j/\lambda)$, where $(\beta_1, \ldots, \beta_{m-1})\leq_c(\alpha_1, \ldots,\alpha_{m-1})$, form a complete residue system modulo $\mathbf s_0/\lambda$.

To show that $(1/\lambda){\rm Ap}(S_n, \mathbf s_0)\subseteq (1/\lambda)T(m)$, let $s\in (1/\lambda){\rm Ap}(S_n, \mathbf s_0)$. By Proposition \ref{prop-residual}, there is a residual $(m-1)-$tuple $(\gamma_1, \ldots, \gamma_{m-1})$ such that $s=\sum_{j=1}^{m-1}\gamma_j(\mathbf s_j/\lambda)$. We want to show that $(\gamma_1, \ldots, \gamma_{m-1})\leq_c(\alpha_1, \ldots, \alpha_{m-1})$. Suppose, on the contrary, that $(\gamma_1, \ldots, \gamma_{m-1})>_c(\alpha_1, \ldots, \alpha_{m-1})$. There exists an $(m-1)-$tuple $(\beta_1, \ldots, \beta_{m-1})$ less than or equal to $(\alpha_1, \ldots, \alpha_{m-1})$ such that 
\begin{equation*}
 s\equiv \sum_{j=1}^{m-1}\beta_j(\mathbf s_j/\lambda)\ ({\rm mod}\ \mathbf s_0/\lambda).
\end{equation*}
If we show that $s=\sum_{j=1}^{m-1}\gamma_j(\mathbf s_j/\lambda)>\sum_{j=1}^{m-1}\beta_j(\mathbf s_j/\lambda)$, we would have a contradiction, since $s\in{\rm Ap}((1/\lambda)S_n, \mathbf s_0/\lambda)$ cannot be congruent modulo $\mathbf s_0/\lambda$ to an element of $(1/\lambda)S_n$ less than $s$. 

Now, if $(\delta_1, \ldots, \delta_{m-1})$ is the immediate successor of $(\beta_1, \ldots, \beta_{m-1})$, then 
\begin{equation*}
    \sum_{j=1}^{m-1}\delta_j(\mathbf s_j/\lambda)-\sum_{j=1}^{m-1}\beta_j(\mathbf s_j/\lambda)=\begin{cases}
    \mathbf s_1/\lambda,&\text{if $\alpha_i<a$ for all $i$};\\
    (\mathbf s_{i+1}/\lambda)-a\mathbf (s_{i}/\lambda),&\text{if $\alpha_i=a$ for some $i$}.
    \end{cases}
\end{equation*}
Note that $(\mathbf s_{i+1}/\lambda)-a(\mathbf s_i/\lambda)=(a-1)d/\lambda=(\mathbf s_1/\lambda)-a(\mathbf s_0/\lambda)$, so 
\begin{equation*}
    \sum_{j=1}^{m-1}\delta_j(\mathbf s_j/\lambda)-\sum_{j=1}^{m-1}\beta_j(\mathbf s_j/\lambda)=\begin{cases}
    \mathbf s_1/\lambda,&\text{if $\alpha_i<a$ for all $i$};\\
    (\mathbf s_{1}/\lambda)-a(\mathbf s_{0}/\lambda),&\text{if $\alpha_i=a$ for some $i$}.
    \end{cases}
\end{equation*}
Let $t$ be the number of residual $r-$tuples $(\delta_1,\ldots,\delta_{m-1})$ such that 
\begin{equation*}
    (\beta_1,\ldots,\beta_{m-1})\leq_c(\delta_1,\ldots,\delta_{m-1})<_c(\gamma_1, \ldots, \gamma_{m-1}),
\end{equation*}
and let $u$ be the number of such $r-$tuples $(\delta_1,\ldots,\delta_{m-1})$ such that $\delta_i=a$ for some $i$. Then, 
\begin{equation*}
    \sum_{j=1}^{m-1}\gamma_j(\mathbf s_j/\lambda)-\sum_{j=1}^{m-1}\beta_j(\mathbf s_j/\lambda)=\frac{(t-u)\mathbf s_1+u(\mathbf s_1-a\mathbf s_0)}{\lambda}=\frac{t\mathbf s_1-au\mathbf s_0}{\lambda}.
\end{equation*}
Since $\mathbf s_0/\lambda$ divides $t\mathbf s_1-au\mathbf s_0$, and $\mathbf s_0/\lambda$ and $\mathbf s_1/\lambda$ are relatively prime, we see that $\mathbf s_0/\lambda$ divides $t$. So, we have $t\geq \mathbf s_0/\lambda$. It is not difficult to show that the number of residual $(m-1)-$tuples that have an entry equal to $a$ is precisely $\mathbf R_{m-1}$. Hence, $u\leq \mathbf R_{m-1}$, and we get 
\begin{equation*}
     ts_1-au\mathbf s_0\geq \mathbf (s_0/\lambda)\mathbf s_1-au\mathbf s_0=\mathbf s_0(\mathbf s_1/\lambda-au)\geq \mathbf s_0(\mathbf R_m-a\mathbf R_{m-1})=\mathbf s_0>0.
\end{equation*}
We conclude that $ \sum_{j=1}^r\gamma_j(\mathbf s_j/\lambda)>\sum_{j=1}^r\beta_j(\mathbf s_j/\lambda)$, which is what we wanted to reach a contradiction. 

We conclude that $(1/\lambda){\rm Ap}(S_n, \mathbf s_0)\subseteq (1/\lambda)T(m)$, and since both sets have $\mathbf s_0/\lambda$ elements, it results that $(1/\lambda){\rm Ap}(S_n, \mathbf s_0)=(1/\lambda)T(m)$ and ${\rm Ap}(S_n, \mathbf s_0)=T(m)$. 
\end{proof}

\section{The Frobenius number}
The following result is a consequence of Theorem \ref{teo-Apery_set_characterization}; it allows us to compute the Frobenius number of $S_n$ under the condition $d>0$. 

\begin{corollary}\label{cor-max_apery}
Assume $d>0$. If $m=e(S_n)$, $\lambda=\gcd(S_n)$ and $(\alpha_1, \ldots, \alpha_{m-1})$ is the unique residual $(m-1)-$tuple such that $\mathbf s_0/\lambda=1+\sum_{j=1}^{m-1}\alpha_j\mathbf R_j$, then
\begin{equation*}
    \max {\rm Ap}(S_n,\mathbf s_0)=\sum_{j=1}^{m-1}\alpha_j\mathbf s_j.
\end{equation*}
Moreover, for the numerical semigroup $(1/\lambda)S_n$, we have 
\begin{equation*}
    {\rm F}((1/\lambda)S_n)=\sum_{j=1}^{m-1}\alpha_j(\mathbf s_j/\lambda)-\mathbf s_0/\lambda.
\end{equation*}
\end{corollary}

\begin{proof}
The condition $\mathbf R_m\leq \mathbf s_1/\lambda$ is satisfied by Proposition \ref{prop-case_d>0}. Then, Theorem \ref{teo-Apery_set_characterization} applies. If $(\beta_1, \ldots, \beta_{m-1})<_c(\alpha_1, \ldots, \alpha_{m-1})$, let $t$ be the number of residual $r-$tuples $(\delta_1,\ldots,\delta_{m-1})$ such that 
\begin{equation*}
    (\beta_1,\ldots,\beta_{m-1})\leq_c(\delta_1,\ldots,\delta_{m-1})<_c(\alpha_1, \ldots, \alpha_{m-1}),
\end{equation*}
and $u$ the number of such $r-$tuples $(\delta_1,\ldots,\delta_{m-1})$ such that $\delta_i=a$ for some $i$. Then, 
\begin{equation*}
    \sum_{j=1}^r\alpha_j\mathbf s_j-\sum_{j=1}^r\beta_j\mathbf s_j=(t-u)\mathbf s_1+u(\mathbf s_1-a\mathbf s_0)=(t-u)\mathbf s_1+u(a-1)d>0,
\end{equation*}
since $t>0, u\leq t$ and $d>0$. This proves that $\max {\rm Ap}(S_n,\mathbf s_0)=\sum_{j=1}^{m-1}\alpha_j\mathbf s_j$.
\end{proof}

\begin{example}\label{ejem-frobenius}
\begin{enumerate}
    \item For $a\geq 2$, consider the generating sequence $x_n=a^{n}-1$. Then, $e(S_n)=n$ and $\lambda=\gcd(S_n)=a-1$. Assume $n>1$. Note that $\mathbf s_0/(a-1)=1+a\mathbf R_{n-1}$. Thus, ${\rm Ap}(S_n, \mathbf s_0)$ is the set of all combinations $\sum_{j=1}^{n-1}\beta_j\mathbf s_j$, where $(\beta_1, \ldots, \beta_{n-1})$ is a residual $(n-1)-$tuple less than or equal to $(0,\ldots, 0,a)$ in the co-lexicographic order. By Corollary \ref{cor-max_apery} we have 
\begin{align*}
    {\rm F}((1/(a-1))S_n)=a\frac{\mathbf s_{n-1}}{a-1}-\frac{\mathbf s_0}{a-1}
    =\frac{a^{2n}-a}{a-1}-\frac{a^{n}-1}{a-1}
    =a^n\mathbf R_{n}-1.
\end{align*}

\item Consider the generating sequence $x_n=5\cdot 3^n-1$. It is easy to find that $\lambda=\gcd(S_n)=2$. Since 
\begin{equation*}
    3^{n+3}-1>2\mathbf s_0=10\cdot 3^n-2=3^{n+2}+3^n-2>3^{n+2}-1,
\end{equation*}
by Proposition \ref{prop-case_d>0}, $e(S_n)=n+3$. Now, $\mathbf s_0/2=(5\cdot 3^n-1)/2$ can be represented as follows: 
\begin{equation*}
    \frac{\mathbf s_0}{2}=1+2\mathbf R_n+\mathbf R_{n+1}, 
\end{equation*}
so by Corollay \ref{cor-max_apery}, we obtain $\max {\rm Ap}(S_n, \mathbf s_0)=2\mathbf s_n+\mathbf s_{n+1}$ and a direct computation shows that 
\begin{equation*}
    {\rm F}((1/2)S_n)=(1/2)(2\mathbf s_n+\mathbf s_{n+1}-\mathbf s_0)=5\cdot 3^n\mathbf R_n+1.
\end{equation*}
\end{enumerate}
\end{example}

\paragraph{Generating sequences of type $x_n=(2^k-1)2^n-1$, where $k\geq 2$.}
We consider the generating sequence $x_n=(2^k-1)2^n-1$, where $k\geq 2$, which was treated in \cite{gu}. In this case, $S_n$ is a numerical semigroup for all $n\geq 1$, and we know that $e(S_n)=n+k$. We also have 
\begin{align*}
\mathbf s_0-1=2^{n+k}-2^n-2=\mathbf R_{n+k}-\mathbf R_n-2.
\end{align*}
Since $\mathbf R_{n+k}=\sum_{j=n}^{n+k-1}\mathbf R_j+\mathbf R_n+k$ (by \eqref{eq-general_relation_repunits}), we get 
\begin{equation*}
    \mathbf s_0-1=\sum_{j=n}^{n+k-1}\mathbf R_j+k-2.
\end{equation*}
If we suppose $2\leq k\leq 2^n$, which is equivalent to $0\leq k-2<\mathbf R_{n}$, then by Corollary \ref{cor-representation_residual}, there is a unique residual $(n-1)-$tuple $(k_1, \ldots, k_{n-1})$ such that $k-2=\sum_{j=1}^{n-1}k_j\mathbf R_j$. Therefore, 
\begin{equation*}
    \mathbf s_0-1=\sum_{j=n}^{n+k-1}\mathbf R_j+k-2=\sum_{j=n}^{n+k-1}\mathbf R_j+\sum_{j=1}^{n-1}k_j\mathbf R_j.
\end{equation*}
This representation was found in \cite{gu}, where they found a characterization of ${\rm Ap}(S_n,\mathbf s_0)$, computed ${\rm F}(S_n)$ and genus of $S_n$. Here, we are interested in computing $\max {\rm Ap}(S_n,\mathbf s_0)$ and ${\rm F}(S_n)$ in the case $k>2^n$, which was not solved in \cite{gu}.  

Assume $k>2^n$. In this case, we cannot represent $k-2$ in the form $k-2=\sum_{j=1}^{n-1}k_j\mathbf R_j$ for any residual $(n-1)-$tuple $(k_1, \ldots, k_{n-1})$. We proceed as follows. By \eqref{eq-general_relation_repunits}, we have $\mathbf R_{n+k-2^n-1}=\sum_{j=n}^{n+k-2^n-2}\mathbf R_j+\mathbf R_{n}+k-2^n-1$, or, equivalently, $k-2=\mathbf R_{n+k-2^n-1}-\sum_{j=n}^{n+k-2^n-2}\mathbf R_j$. Then,
\begin{align*}
    \sum_{j=n}^{n+k-1}\mathbf R_j+k-2&=\sum_{j=n}^{n+k-1}\mathbf R_j+\mathbf R_{n+k-2^n-1}-\sum_{j=n}^{n+k-2^n-2}\mathbf R_j\\
    &=\sum_{j=n+k-2^n}^{n+k-1}\mathbf R_j+2\mathbf R_{n+k-2^n-1}.
\end{align*}
Thus, the residual $(n+k-1)-$tuple that realizes $\max {\rm Ap}(S_n, \mathbf s_0)$ is $(0,\ldots,0,2,1,\ldots, 1)$, where the 2 is in the $(n+k-2^n-1)-$th coordinate. It follows that if $k>2^n$, then 
\begin{equation*}
    \max {\rm Ap}(S_n, \mathbf s_0)=2\mathbf s_{n+k-2^n-1}+\sum_{j=n+k-2^n}^{n+k-1}\mathbf s_{j}.
\end{equation*}
By \eqref{eq-general_relation_s_j}, we have $2\mathbf s_{n+k-2^n-1}+\sum_{j=n+k-2^n}^{n+k-1}\mathbf s_{j}=\mathbf s_{n+k}-2^n-1$, so we can compute the Frobenius number of $S_n$ as follows: 
\begin{align*}
    {\rm F}(S_n)&=\mathbf s_{n+k}-2^n-1-\mathbf s_0\\
    &=\mathbf s_{n+k}-2^n-1-2^{n+k}+2^n+1\\
    &=\mathbf s_{n+k}-2^{n+k}\\
    &=2^{2n+2k}-2^{2n+k}-2^{n+k}-1.
\end{align*}

\paragraph{Generating sequences of type $x_n=(a^k-1)a^n-1$.}
A natural generalization of the sequence treated in \cite{gu} is the generating sequence $x_n=(a^k-1)a^n-1$, where $a\geq 3$, $k\geq 1$. Here, $S_n$ is a numerical semigroup for all $n\geq 1$ and we know that $e(S_n)=n+k+1$. We will omit most of the calculations to compute the Frobenius number of $S_n$. By using \eqref{eq-general_relation_repunits} we obtain the following formula that hepls in finding the maximum of ${\rm Ap}(S_n, \mathbf s_0)$: 
\begin{equation*}
    \mathbf s_0-1
    =(a-2)\mathbf R_{n+k}+(a-1)\sum_{j=n+1}^{n+k-1}\mathbf R_j+\mathbf R_n+k-2
\end{equation*}
There are some cases: 
\paragraph{The case $k=1$.} If $n=1$, it is found that ${\rm F}(S_1)=(a-2)\mathbf s_2-\mathbf s_0=a^5-3a^4+2a^3-a^2+3$. Now, if $n>1$, then 
\begin{equation*}
    {\rm F}(S_n)=a^{2n+3}-3a^{2n+2}+3a^{2n+1}-a^{2n}-a^{n+1}+a^n-2a+3.
\end{equation*}

\paragraph{The case $n=1$ and $k>1$.} In this case 
\begin{align*}
    \mathbf s_0-1&=(a-2)\mathbf R_{k+1}+(a-1)\sum_{j=2}^{k}\mathbf R_j+\mathbf R_1+k-2\\
    &=(a-2)\mathbf R_{k+1}+(a-1)\sum_{j=2}^{k}\mathbf R_j+k-1.
\end{align*}
If $k\leq a+1$, then the Frobenius number of $S_1$ is given by: 
\begin{equation*}
    {\rm F}(S_n)=a^{2k+3}-a^{2k+2}-2a^{k+3}+ka^{k+2}-a^{k+1}+a^3-(k-1)a^2-(k-1)a+3.
\end{equation*}
Now, if $k>a+1$, then 
the Frobenius number of $S_1$ is given by 
\begin{equation*}
    {\rm F}(S_1)=a^{2k+3}-a^{2k+2}-a^{k+3}+a^{k+2}-a^{k+1}-a^2+3.
\end{equation*}
\paragraph{The case $n>1$ and $k\geq 2$.} Here we have two cases. In the first one, $0\leq k-2<(a-1)\mathbf R_n$.
By the division algorithm there exist unique integers $q$ and $r$ such that $0\leq q<a-1$, $0\leq r<\mathbf R_n$, and $k-2=q\mathbf R_n+r$. Then  
\begin{equation*}
\mathbf s_0-1=(a-2)\mathbf R_{n+k}+(a-1)\sum_{j=n+1}^{n+k-1}\mathbf R_j+(q+1)\mathbf R_n+r.   
\end{equation*}
Note that $q+1<a$. Also, there is a unique residual $(n-1)-$tuple $(t_1, \ldots, t_{n-1})$ such that $r=\sum_{j=1}^{n-1}t_j\mathbf R_j$. So the residual $(n+k)-$tuple that realizes $\max {\rm Ap}(S_n, \mathbf s_0)$ is $(t_1, \ldots, t_{n-1}, q+1, a-1,\ldots, a-1, a-2)$ and 
\begin{equation*}
\max {\rm Ap}(S_n, \mathbf s_0)=(a-2)\mathbf s_{n+k}+(a-1)\sum_{j=n+1}^{n+k-1}\mathbf s_j+(q+1)\mathbf s_n+\sum_{j=1}^{n-1}t_j\mathbf s_j.   
\end{equation*}
The Frobenius number results by subtracting $\mathbf s_0$ from $\max {\rm Ap}(S_n, \mathbf s_0)$, which we do not write down explicitly. 

The only remaining case is when $k-2\geq (a-1)\mathbf R_n$, and it can be shown that 
\[
    {\rm F}(S_n)=
    a^{2n+2k+1}-a^{2n+2k}-a^{2n+k+1}-a^{2n+k}-a^{n+k}-a^{n+1}+2a^n-2a+3.
\]

\section{The genus}
To calculate the genus of $S_n$ directly by means of the formula ${\rm g}(S)=\frac{1}{x}\sum_{w\in {\rm Ap}(S,x)}w-\frac{x-1}{2}$, it is useful to compute sums of the form 
\begin{equation*}
    \sum_{(\alpha_1,\ldots, \alpha_r)}\sum_{j=1}^{r}\alpha_j\mathbf s_j,
\end{equation*}
where the sum is taken over all residual $r-$tuples $(\alpha_1,\ldots, \alpha_r)$. The following result gives us a nice formula for this kind of sums. 

\begin{proposition}\label{prop-sum_genus}
For all $r\geq 1$, we have
\begin{equation*}
\sum_{(\alpha_1,\ldots, \alpha_r)}\sum_{j=1}^{r}\alpha_j\mathbf s_j=\frac{1}{2}\left(a\mathbf R_r\mathbf s_{r+1}+ra^{r+1}\mathbf s_0\right)
\end{equation*}
\end{proposition}

\begin{proof}
First, we have
\begin{equation*}
\sum_{(\alpha_1,\ldots, \alpha_r)}\sum_{j=1}^{r}\alpha_j\mathbf s_j=\frac{1}{2}\sum_{j=1}^{r}\left(a^{r+1}+a^{r+1-j}\right)\mathbf s_j.
\end{equation*}
This formula follows exactly as in the proof of Theorem 25 in \cite{repunit} where, although that proof is for the generating sequence of repunit numbers $\mathbf R_n=(a^n-1)/(a-1)$, it works for a general generating sequence. Now, \begin{align*}
    \sum_{j=1}^{r}\left(a^{r+1}+a^{r+1-j}\right)\mathbf s_j&=a^{r+1}\sum_{j=1}^{r}(ca^{n+j}-d)+\sum_{j=1}^ra^{r+1-j}(ca^{n+j}-d)\\
    &=a^{r+1}\left(ca^{n+1}\sum_{j=1}^{r}a^{j-1}-rd\right)+\left(rca^{r+n+1}-da\sum_{j=1}^{r}a^{r-j}\right)\\
    &=ca^{r+n+2}\mathbf R_{r}-rda^{r+1}+rca^{r+n+1}-da\mathbf R_r\\
    &=a\mathbf R_{r}\mathbf s_{r+1}+ra^{r+1}\mathbf s_0
\end{align*}
\end{proof}

\begin{example}
Consider the generating sequence $x_n=5\cdot 3^n-1$. To find the genus of $S_n$, first we calculate the sum $\sum_{(\beta_1,\ldots, \beta_{n+2})}\left(\sum_{j=1}^{n+2}\beta_j\mathbf s_j\right)$ taken over all residual $(n+2)-$tuples $(\beta_1,\ldots, \beta_{n+2})$ which are less than or equal to $(0,\ldots, 0, 2,1,0)$ with respect to the colexicographic order:
\begin{align*}
    \sum_{(\beta_1,\ldots, \beta_{n+2})}\left(\sum_{j=1}^{n+2}\beta_j\mathbf s_j\right)&=\sum_{(\beta_1,\ldots, \beta_{n+1})}\left(\sum_{j=1}^{n+1}\beta_j\mathbf s_j\right)\\
    &=\sum_{(\beta_1,\ldots, \beta_{n})}\left(\sum_{j=1}^{n}\beta_j\mathbf s_j\right)+\sum_{(\beta_1,\ldots, \beta_{n})}\left(\sum_{j=1}^{n}\beta_j\mathbf s_j+\mathbf s_{n+1}\right)\\
    &=\sum_{(\beta_1,\ldots, \beta_{n})}\left(\sum_{j=1}^{n}\beta_j\mathbf s_j\right)+\sum_{(\beta_1,\ldots, \beta_{n-1})}\left(\sum_{j=1}^{n-1}\beta_j\mathbf s_j+\mathbf s_{n+1}\right)\\
    & +\sum_{(\beta_1,\ldots, \beta_{n-1})}\left(\sum_{j=1}^{n-1}\beta_j\mathbf s_j+\mathbf s_n+\mathbf s_{n+1}\right)+(2\mathbf s_n+\mathbf s_{n+1}).\\
\end{align*}
In the last expression, the sums are taken over all residual tuples, without restriction. By Applying Proposition \ref{prop-sum_genus}, we obtain
\begin{align*}
    \sum_{(\beta_1,\ldots, \beta_{n+2})}\left(\sum_{j=1}^{n+2}\beta_j\mathbf s_j\right)&=\frac{1}{2}\left(3\mathbf R_n\mathbf s_{n+1}+n3^{n+1}\mathbf s_0\right)+\frac{1}{2}\left(3\mathbf R_{n-1}\mathbf s_n+(n-1)3^{n}\mathbf s_0\right)\\
    &+\mathbf R_n\mathbf s_{n+1}+\frac{1}{2}\left(3\mathbf R_{n-1}\mathbf s_n+(n-1)3^{n}\mathbf s_0\right)+\mathbf R_n(\mathbf s_n+\mathbf s_{n+1})\\
    &+(2\mathbf s_n+\mathbf s_{n+1})\\
    &=\frac{1}{2}\left((5n+6+25\mathbf R_n)3^n+5\mathbf R_n+1\right)\mathbf s_0.
\end{align*}
Finally, we have 
\begin{align*}
    g((1/2)S_n)&=\frac{1}{(\mathbf s_0/2)}\sum_{(\beta_1,\ldots, \beta_{n+2})}\left(\sum_{j=1}^{n+2}\beta_j(\mathbf s_j/2)\right)-\frac{\mathbf s_0/2-1}{2}\\
    &=\frac{1}{2}\left((5n+6+25\mathbf R_n)3^n+5\mathbf R_n+1\right)-\frac{\mathbf s_0/2-1}{2}\\
    &=\frac{1}{2}(5n+6+25\mathbf R_n)3^n.
\end{align*}
\end{example}

This way of computing the genus of $S_n$ is tedious in general. A recursive way of computing the sum $\sum_{(\beta_j)}\sum_{j=1}^{m-1}\beta_j\mathbf s_j$ over all residual $(m-1)-$tuples $(\beta_1, \ldots, \beta_{m-1})$, less than or equal to $(\alpha_1, \ldots, \alpha_{m-1})$ in the colexicographic order, can be described as follows: Given the residual $(m-1)-$tuple $(\alpha_1, \ldots, \alpha_{m-1})$ under the conditions of Theorem \ref{teo-Apery_set_characterization}, define $G_t$ for $t=1, \ldots, m-1$ as follows:
\begin{equation*}
    G_t=\sum_{(\beta_j)}\sum_{j=1}^t\beta_j\mathbf s_j,
\end{equation*}
where the left-most sum is taken over all residual $t-$tuples $(\beta_1, \ldots, \beta_t)$ such that $(\beta_1, \ldots, \beta_t)\leq_c(\alpha_1, \ldots, \alpha_t)$. It can be shown that 
\begin{equation*}
    G_t=\sum_{j=1}^t\alpha_j\mathbf s_j+\mathbf s_0\frac{\alpha_t(\alpha_t-1)}{2}\left(1+\sum_{j=1}^{t-1}\alpha_j\mathbf R_j\right)+\alpha_tG_{t-1},
\end{equation*}
and this can be used to find ${\rm g}(S_n)$ recursively. 
\section{Pseudo-Frobenius numbers and type}

Let $m=e(S_n), \lambda=\gcd(S_n)$ and assume that $\mathbf R_m\leq \mathbf s_1/\lambda$. Theorem \ref{teo-Apery_set_characterization} gives us a characterization of ${\rm Ap}(S_n, \mathbf s_0)$. If $(\alpha_1, \ldots, \alpha_{m-1})$ is the only residual $(m-1)-$tuple such that $\mathbf s_0/\lambda-1=\sum_{j=1}^{m-1}\alpha_j\mathbf R_j$, let $B(m-1)$ be the set of all residual $(m-1)$-tuples $(\beta_1, \ldots, \beta_{m-1})$ that are less than or equal to $(\alpha_1, \ldots, \alpha_{m-1})$ with respect to the colexicographic order. Then, the elements of ${\rm Ap}(S_n, \mathbf s_0)$ are precisely those elements in $S_n$ of the form $\sum_{j=1}^{m-1}\beta_j\mathbf s_j$, where $(\beta_1, \ldots, \beta_{m-1})\in B(m-1)$. 

To find the pseudo-Frobenius numbers of $(1/\lambda)S_n$, we must find first the maximal elements of ${\rm Ap}(S_n,\mathbf s_0)$, with respect to the order $\leq_{S_n}$ given by $a\leq_{S_n}b$ if and only if $b-a\in S_n$. Then, the pseudo-Frobenius numbers of $(1/\lambda)S_n$ are those integers of the form $m/\lambda-\mathbf s_0/\lambda$, where $m$ is maximal in ${\rm Ap}(S_n,\mathbf s_0)$, with respect to the order $\leq_{S_n}$.

First, we find the maximal $(m-1)-$tuples $(\beta_1, \ldots, \beta_{m-1})\in B(m-1)$ with respect to the \textit{product order}. We describe these $(m-1)-$tuples as follows: Let $j_1<\cdots<j_r$ the only indexes $j\in \{2,\ldots, m-1\}$ such that $\alpha_j>0$. The maximal $(m-1)-$tuples in $B(m-1)$ with respect to the product order can be partitioned into $r+1$ subsets of $B(m-1)$, $M_0, M_1,\ldots, M_r$, as we describe next. For $j\in \{j_1, \ldots, j_r\}$, let 
\begin{align*}
    M_j=\{&(a,a-1,\ldots, a-1,\alpha_{j}-1, \alpha_{j+1},\ldots, \alpha_{m-1}),\\
    &(0,a,a-1,\ldots, a-1,\alpha_{j}-1,\alpha_{j+1},\ldots, \alpha_{m-1}),\\
    &(0,0,a,a-1,\ldots, a-1,\alpha_{j}-1, \alpha_{j+1},\ldots, \alpha_{m-1}),\\
    & \ \ \vdots\\
    &(0,\ldots, 0,a,\alpha_{j_t}-1, \alpha_{j_t+1},\ldots, \alpha_{m-1})\},
\end{align*}
and let $M_0=\{(\alpha_1, \ldots, \alpha_{m-1})\}$. As we see, the set $M_0$ has 1 element and $M_j$ has $j-1$ elements for each $j\in\{j_1,\ldots,j_r\}$.

For $j\in\{j_1,\ldots,j_r\}$, each $(m-1)-$tuple in $M_j$ generates an element of ${\rm Ap}(S_n,\mathbf s_0)$. Namely, the tuple $(0,\ldots, 0, a, a-1, \ldots, a-1, \alpha_{j}-1, \alpha_{j+1}, \ldots, \alpha_{m-1})$, where the $a$ is in position $i$, $1\leq i<j$, corresponds to the element 
\begin{align*}
    \mathbf f_{i,j}&:=a\mathbf s_i+(a-1)\sum_{l=i+1}^{j-1}\mathbf s_l+(\alpha_{j}-1)\mathbf s_{j}+\sum_{l=j+1}^{m-1}\alpha_l\mathbf s_l\\
    &=a\mathbf s_{j-1}-(j-i-1)(a-1)d+(\alpha_{j}-1)\mathbf s_{j}+\sum_{l=j+1}^{m-1}\alpha_l\mathbf s_l\\
    &=\alpha_{j}\mathbf s_{j}+\sum_{l=j+1}^{m-1}\alpha_l\mathbf s_l-(j-i)(a-1)d\\
    &=\sum_{l=j}^{m-1}\alpha_l\mathbf s_l-(j-i)(a-1)d.
\end{align*}
Also, the tuple $(\alpha_1, \ldots, \alpha_{m-1})$ corresponds to 
\begin{equation*}
    \mathbf f_{0}:=\sum_{l=1}^{m-1}\alpha_l\mathbf s_l.
\end{equation*}
This way, the candidates to maximal elements of ${\rm Ap}(S_n, \mathbf s_0)$ with respect to $\leq_{S_n}$ are the elements $\mathbf f_{i,j}$, where $j\in\{j_1,\ldots,j_r\}, 1\leq i<j$, along with the element $\mathbf f_{0}$. 

Now, assume that $1\leq t<r$. For $1\leq i<j_{t}$, we have 
\begin{align*}
    \mathbf f_{i,j_{t}}-\mathbf f_{i+j_{t+1}-j_t,j_{t+1}}&=\sum_{l=j_t}^{m-1}\alpha_l\mathbf s_l-(j_t-i)(a-1)d-\sum_{l=j_{t+1}}^{m-1}\alpha_l\mathbf s_l+(j_{t}-i)(a-1)d\\
    &=\sum_{l=j_t}^{j_{t+1}-1}\alpha_l\mathbf s_l\in S_n.
\end{align*}
This shows that $\mathbf f_{j_{t+1},i+j_{t+1}-j_t}\leq_{S_n}\mathbf f_{j_{t},i}$. If we set $j_0=1$, the candidates to pseudo-Frobenius numbers of $S_n$ reduce to the elements $\mathbf f_{i,j_t}$, $t\in \{1, \ldots,r\}$, $j_{t-1}\leq i<j_t$, along with the element $\mathbf f_0$. 

\begin{theorem}\label{teo-pseudo}
Let $m=e(S_n), \lambda=\gcd(S_n)$ and assume that $\mathbf R_m\leq \mathbf s_1/\lambda$. If $(\alpha_1, \ldots, \alpha_{m-1})$ is the only residual $(m-1)-$tuple such that $\mathbf s_0/\lambda-1=\sum_{j=1}^{m-1}\alpha_j\mathbf R_j$, then 
\begin{equation*}
    {\rm PF}(S_n)\subseteq\{\mathbf f_0\}\cup\{\mathbf f_{i,j_t}: t\in \{1,\ldots, r\}, j_{t-1}\leq i<j_t\}.
\end{equation*}
In particular, the type of $S_n$ satisfies the inequality ${\rm t}(S_n)\leq e(S_n)-1$.
\end{theorem}

\begin{proof}
It only remains to prove that ${\rm t}(S_n)\leq e(S_n)-1$. Now, for $t\in\{1,\ldots, r\}$, there are
$j_{t}-j_{t-1}$ elements of the form $\mathbf f_{i,j_t}$, $j_{t-1}\leq i<j_t$. Then, the number of candidates to pseudo-Frobenius numbers of $S_n$ is 
\begin{equation*}
    1+\sum_{t=1}^{r}(j_t-j_{t-1})=j_r\leq m-1=e(S_n)-1.
\end{equation*}
This ends the proof.
\end{proof}

The equality ${\rm t}(S_n)=e(S_n)-1$ holds in many cases, for instance, for the semigroups considered in \cite{mersenne, repunit, thabit}. Moreover, as we see in the proof of Theorem \ref{teo-pseudo}, ${\rm t}(S_n)\leq j_r$, where $j_r$ is the greatest integer $j$ in $\{1,\ldots, m-1\}$ such that $\alpha_j>0$. For instance, in part 2 of Example \ref{ejem-frobenius}, for $\mathbf s_0=5\cdot 3^n-1$ we found that $e(S_n)=n+3$, while $\mathbf s_0/\lambda-1=2\mathbf R_n+\mathbf R_{n+1}$, which implies that ${\rm t}(S_n)\leq n+1$. This example shows that the type of $S_n$ can be strictly less than $e(S_n)-1$.

Finally, a conjecture of Wilf on numerical semigroups establishes that for any numerical semigroup $S$, if $n(S)$ is the number of elements in $S$ less than ${\rm F}(S)$, then ${\rm F}(S)+1\leq n(S)\cdot e(S)$. This conjectures has been verified for various families of numerical semigroups, see \cite{delgado}. Here, we can prove that Wilf's conjecture is true for the numerical semigroup $(1/\lambda)S_n$, under the hypothesis of Theorem \ref{teo-pseudo}. In fact, this follows by \cite[Theorem 20]{froberg}, which states the inequality ${\rm F}(S)+1\leq n(S)({\rm t}(S)+1)$ for any numerical semigroup $S$. In our case, we have ${\rm t}(S_n)+1\leq e(S_n)$, so Wilf's conjecture follows trivially for $S_n$, under the hypothesis of Theorem \ref{teo-pseudo}.

\end{document}